\documentclass[11pt]{amsart}

\usepackage{amssymb,amsmath,amsthm,amscd}
\usepackage{latexsym, amsxtra, mathrsfs}
\usepackage[svgnames]{xcolor}
\usepackage{hyperref}
\hypersetup{colorlinks,%
    citecolor=Blue,%
    linkcolor=Purple,%
}

\advance\oddsidemargin by -1cm
\advance\evensidemargin by -1cm 
\newtheorem{theorem}{Theorem}

\newtheorem{proposition}{Proposition}
\newtheorem{corollary}{Corollary}
\newtheorem{lemma}{Lemma}
\newtheorem{definition}{Definition}

\theoremstyle{definition}
\newtheorem{remark}{Remark}

\newcommand{\bdm}{\begin{displaymath}}
\newcommand{\edm}{\end{displaymath}}
\newcommand{\bq}{\begin{equation}}
\newcommand{\eq}{\end{equation}}
\newcommand{\bqn}{\begin{equation*}}
\newcommand{\eqn}{\end{equation*}}

\newcommand{\eps}{\varepsilon}

\newcommand{\mklm}[1]{\left\{ #1 \right\}}
\newcommand{\eklm}[1]{\left\langle  #1 \right\rangle}

\renewcommand{\d}{\,d}
\newcommand{\N}{{\mathbb N}}

\newcommand{\C}{{\mathbb C}}
\newcommand{\R}{{\mathbb R}}
\newcommand{\A}{{\mathcal A}}

\newcommand{\D}{{\mathcal D}}

\newcommand{\F}{{\mathcal F}}

\renewcommand{\H}{{\mathcal H}}
\newcommand{\Ucal}{{\mathcal U}}

\renewcommand{\epsilon}{\varepsilon}
\renewcommand{\phi}{\varphi}
\renewcommand{\rho}{\varrho}
\newcommand{\1}{{ \bf  1}}

\newcommand{\Cinft}{{\rm C^{\infty}}}
\newcommand{\CT}{{\rm C^{\infty}_c}}

\renewcommand{\L}{{\rm L}}

\newcommand{\Ncal}{{\mathcal N}}

\renewcommand{\S}{{\mathcal S}}

\newcommand{\Sym}{{\rm S}}
\newcommand{\Syms}{{\rm S^{-\infty}}}

\newcommand{\GL}{\mathrm{GL}}

\newcommand{\g}{{\bf \mathfrak g}}

\renewcommand{\a}{{\bf\mathfrak a}}

\newcommand{\n}{{\bf\mathfrak n}}

\newcommand{\U}{{\mathfrak U}}

\newcommand{\ad}{\mathrm{ad}\,}

\renewcommand{\det}{\mathrm{det}\,}

\renewcommand{\Re}{\mathrm{Re}\,}

\newcommand{\Fix}{\mathrm{Fix}}

\DeclareMathOperator{\supp}{supp}

\DeclareMathOperator{\tr}{tr}
\DeclareMathOperator{\gd}{\partial}

\DeclareMathOperator{\Tr}{Tr}

\newcommand{\e}[1]{\,{\mathrm e}^{#1}\,}
\newcommand{\dbar}{{\,\raisebox{-.1ex}{\={}}\!\!\!\!d}}

\begin{document}

\author{St\'ephanie Cupit-Foutou, Aprameyan Parthasarathy, Pablo Ramacher}
\title[Microlocal analysis on Wonderful Varieties]{Microlocal analysis on Wonderful varieties.\\  Regularized traces and global characters }
\thanks{This first-named author was  funded  by the SFB/TR 12 of the German Research Foundation (DFG).
The two last authors would like to thank Yuri Tschinkel 
for seminal conversations about the subject of this paper during their stay at the Mathematical Institute of G\"ottingen University.}
\date{November 9, 2015}

\begin{abstract}
Let $\mathbf{G}$ be a connected reductive complex algebraic group with split real form $(G,\sigma)$.
Consider a strict wonderful $\mathbf{G}$-variety $\bf{X}$ equipped with its $\sigma$-equivariant real structure, 
and let $X$ be the corresponding real locus. Further, let $E$ be a real differentiable $G$-vector bundle over $X$. 
In this paper, we introduce a distribution character  for the regular representation  of $G$ on  the space of smooth sections of $E$, 
and show that on a certain open subset of $G$ of transversal elements it is locally integrable and given by a sum over fixed points. 
\end{abstract}

\maketitle

\tableofcontents

\section{Introduction}

Let $G$ be a real reductive group. 
In classical harmonic analysis  a crucial role is played by the global character of an irreducible admissible  representation $(\nu,\H)$ of $G$ on a Hilbert space $\H$.
It is a distribution $\Theta_\nu:\mathscr{S}(G) \ni f \mapsto \tr \nu(f)\in \C$ on the group given in terms of the trace of the convolution operators
\bqn 
\nu(f)=\int_{G} f(g) \nu(g) \d_{G}(g),
\eqn
where $\mathscr{S}(G)$ denotes certain space of rapidly decaying functions on $G$ and $d_G$ a Haar measure on $G$ \cite[Chapter 8]{wallach}.
By Harish-Chandra's regularity theorem, $\Theta_\nu$ is known to be locally integrable,
and represents a natural generalization of the character of a finite-dimensional representation.
As a consequence of this theorem,  Harish-Chandra  was able to characterize  tempered representations in terms of the growth properties of their global characters,
and fully determine the irreducible $L^2$-integrable representations of $G$. For  a parabolically induced representation,
Atiyah and Bott \cite{atiyah-bott68} interpreted $\Theta_\nu$  in terms of a  fixed point formula,
extending the classical Lefschetz fixed point theorem to geometric endomorphisms on elliptic complexes via pseudodifferential operators. 

Let now $G$ be the split real form of a connected reductive complex algebraic group $\mathbf{G}$ and
$\sigma$  the corresponding anti-holomorphic involution of $\mathbf{G}$.
Consider further the real locus $X$ of a strict wonderful complex algebraic $\mathbf{G}$-variety $\bf{X}$ 
equipped with its {canonical} $\sigma$-equivariant  real structure.
The variety $X$ is, in particular, a projective real algebraic $G_0$-variety,   $G_0$ being  the identity component of $G$.
In this paper, we introduce a  distribution character $\Theta_\pi$ for the regular representation $(\pi,\Cinft(X,E))$ of $G_0$ on the space of smooth sections of 
 a given real differentiable $G_0$-vector bundle $E$ over $X$.
Since the $G$-action on $X$ is not transitive, the corresponding convolution operators $\pi(f)$ are not smoothing,
so that a trace can no longer be defined by restricting their kernels to the diagonal.
Instead,  we show that  the operators $\pi(f)$ can be characterized as totally characteristic pseudodifferential operators (Theorem \ref{thm:2503}), based on   a certain integral transform which we now describe.
 As was shown in  \cite{ACF12},  there is a canonically defined chart on $X$ given by the local structure of $X$ around its unique closed $G_0$-orbit. 
 More specifically, after fixing a Borel subgroup $B\subset G$ and a maximal torus $T\subset B$, consider
  the standard parabolic subgroup $P\supset B$ of $G$ such that $G_0/P_0$ is isomorphic to the closed $G_0$-orbit of $X$,
and let $P=P^uL$ be its Levi decomposition with $T\subset L$.
Then, the aforementioned  canonical chart is given by  $P^u \cdot Z\simeq P^ u\times Z$ where $Z$ is an affine $L$-subvariety of $X$ isomorphic to $\mathbb R^r$, and  $r$ denotes  the rank of $X$.
We now introduce the mentioned integral transform  in terms of the spherical roots of $\bf X$   as a    mapping (Definition \ref{def:spherfourier} and Proposition \ref{prop:spherfourier})
\bqn
\F_{\mathrm{spher}}:\S(P_0) \longrightarrow S^{-\infty}(P^u \cdot Z^\ast \times \R^{s+r})
\eqn
from the Casselman-Wallach space $\S(P_0)$ of rapidly decreasing functions on $P_0$ to the space of smoothing symbols $S^{-\infty}(P^u\cdot Z^\ast \times \R^{s+r})$, where $Z^\ast :=\mklm{z \in Z: z_1 \cdots z_r\not=0}$ and $s$ denotes the dimension of $P^u$.

As a consequence, the microlocal description of the convolution operators $\pi(f)$   allows us to  characterize the singular nature of their kernels,  and to introduce a regularized trace $\Tr_{reg} \pi(f)$ for them.  This yields a distribution  on $G_0$ which is given by  (Theorem \ref{prop:tracereg})
$$
\Theta_\pi:\CT(G_0) \ni f \mapsto \Tr_{reg} \pi(f)\in \C,
$$ 
 $\CT(G_0)$ denoting the space of smooth functions on $G_0$ with compact support. We call $\Theta_\pi$ the \emph{global character} of the representation $\pi$. It should be emphasized that this distribution is given in terms of the spherical roots of ${\bf X}$, and therefore encodes a large part of the structure of $X$. We then prove that on a certain open set of transversal elements $G_0(X)\subset G_0$ the distribution  $\Theta_\pi$ is locally integrable  and given by (Theorem \ref{thm:FPF})
\bqn 
\Theta_\pi(f)=\int_{G_0(X)} f(g) \Tr^\flat \pi(g) d_G(g), \qquad f \in \CT(G_0(X)),
\eqn
$d_G$ being a Haar measure on $G$ and $\Tr^\flat \pi(g)$  the flat trace of $\pi(g)$.
Writing  $\Phi_g(x):=g^{-1} \cdot x$, the latter can be expressed as a sum over fixed points 
\bqn 
\Tr^\flat \pi(g)=  \sum_{x \in \Fix(X,g)}     \frac{\Tr \, (g:E_x \rightarrow E_x) }{ |\det (\1-d\Phi_{g}(x))|}
\eqn
which  is manifestly   invariant under conjugation.\\

The  global characters $\Theta_\pi$ introduced in this paper are expected to be relevant in the context of harmonic and global analysis on spherical varieties,
and  the authors intend to study their properties in detail in the future.  One of the questions to be dealt with is  the invariance under conjugation of  the distributions $\Theta_\pi$ on the entire group $G_0$. 
Further, it would be interesting  to compare the regularized trace introduced in this paper with other existing trace concepts, to compute it for specific representations $(\pi,\Cinft(X,E))$, and to examine the distributions $\Theta_\pi$ in light of the Harish-Chandra theory of invariant eigendistributions. We also intend to characterize the representations $(\pi,\Cinft(X,E))$ and their wave front sets  more closely  for specific vector bundles. In a  different direction, it would also be interesting to study  the resolvent of a strongly elliptic operator on $X$ together with its meromorphic continuation,  and to examine the action of discrete subgroups on $X$ and associated representations. 

This paper is based on the local structure theorem for  real loci of 
strict wonderful $\mathbf{G}$-varieties recently proved by Akhiezer and Cupit-Foutou \cite{ACF12},
and  generalizes results already obtained by Parthasarathy and Ramacher \cite{parthasarathy-ramacher14} for the Oshima compactification of a Riemannian symmetric space,
as well as earlier work of Ramacher \cite{ramacher06}.
Actually, the mentioned local structure theorem provides a natural framework in which the  previous results  can be understood in a conceptually  simple way.

\section{Wonderful varieties}
\label{sec:wond.var}

Throughout this article we shall adopt the convention of writing complex objects with boldface letters and the corresponding real objects with ordinary ones.
Let $G$ be the split real form of a connected reductive complex algebraic group $\mathbf{G}$ of rank $l$,
and let $\sigma:\mathbf{G}\rightarrow \mathbf{G}$ be the anti-holomorphic involution defining the split real form $G$, 
so that $G=\mathbf{G}^{\sigma}=\mklm{g \in \mathbf{G}: \sigma(g)=g}$.
In particular, $G$ is a real reductive group. Since $G$ is   not necessarily connected, denote by  $G_0$ the identity component of $G$.
Fix a maximal algebraic torus $\mathbf{T}\simeq(\C^\ast)^l$ of $\mathbf{G}$ and a Borel subgroup $\mathbf{B}$ of $\mathbf{G}$ containing it. 
We further assume that $\mathbf{B}$  and $\mathbf{T}$ are both $\sigma$-stable. We recall the definition of a wonderful $\mathbf{G}$-variety. 

\begin{definition}\emph{(\cite{luna96})}
An algebraic $\mathbf{G}$-variety $\mathbf{X}$ is called \emph{wonderful of rank $r$} if
\begin{enumerate}
\item $\mathbf{X}$ is projective and smooth;
\item $\mathbf{X}$ admits an 
open $\mathbf{G}$-orbit whose complement 
consists of a finite union of smooth prime divisors $\mathbf{X}_1,\ldots, \mathbf{X}_r$ with
normal crossings;
\item the $\mathbf{G}$-orbit closures of $\mathbf{X}$ 
are given by the partial intersections of the $\mathbf{X}_i$.
\end{enumerate}
\end{definition}

In particular, notice that $\mathbf{X}$ has a unique closed, hence projective $\mathbf{G}$-orbit, 
given by the intersection of all prime divisors ${\bf X}_i$. 
Moreover, the $\mathbf G$-variety $\mathbf X$ is spherical that is, $\mathbf B$ has an open orbit in $\mathbf X$ \cite{luna96}.

Further, recall that a  real structure on $\mathbf{X}$  is an involutive anti-holomorphic map 
$\mu:\mathbf{X} \rightarrow \mathbf{X}$; 
it is said to be \emph{$\sigma$-equivariant} if $\mu(g\cdot x)=\sigma(g)\cdot\mu(x)$ for all $(g,x)\in \mathbf{G}\times \mathbf{X}$.
Crucial for the ensuing analysis is the existence of a unique $\sigma$-equivariant real structure (called canonical) on some wonderful varieties. 
Wonderful varieties whose points have self-normalizing stabilizers are called \emph{strict}.
For such varieties one has the following 

\begin{theorem}\cite{ACF12}
Let  $\mathbf{X}$ be a strict wonderful $\mathbf{G}$-variety. Then
\begin{enumerate}
 \item there exists a unique $\sigma$-equivariant real structure $\mu$ on $\mathbf{X}$; 
\item the real locus $X$  of $(\mathbf{X},\mu)$ is not empty, and constitutes a smooth projective real algebraic $G_0$-variety with finitely many $G_0$-orbits,
finitely many open $B$-orbits and a unique closed $G_0$-orbit.
\end{enumerate}
\end{theorem}
\qed

Examples of real loci of strict wonderful varieties include the Oshima-Sekiguchi compactification of a Riemannian symmetric space;
such a compactification can be realized as the real locus of the De-Concini-Procesi wonderful compactification of the complexification of the given symmetric space,
up to a finite quotient, see \cite[Chapter 8, Section II.14]{borel-ji}.

In what follows, let  $X$ be the real locus of a strict wonderful $\mathbf{G}$-variety of rank $r$ equipped with its canonical $\sigma$-equivariant real structure.
Let $Y$ denote the  unique closed $G_0$-orbit of $X$, 
and consider the  parabolic subgroup $P={\bf P}^\sigma\supset B={\bf B}^\sigma$ of $G$ such that $Y\simeq G_0/P_0$.
Let $P=P^uL$ be the Levi decomposition of $P$ with $T={\bf T}^\sigma\subset L$.
Notice that $P^u$ is connected. 

The  following local structure theorem describes the structure of the real locus $X$ locally around $Y$, and will be essential for everything that follows.
It constitutes the real analogue of Theorem 1.4 in \cite{BLV86}.

\begin{theorem}\cite[Section 5]{ACF12}
\label{thm:ACF}
There exists a real algebraic $L$-subvariety $Z$ of $X$ such that
\begin{enumerate}
\item the natural mapping 
\bqn
P^u\times  Z\rightarrow P^u
\cdot Z
\eqn
is a $P^u$-equivariant isomorphism;
\item
each $G_0$-orbit in $X$ contains points of  $Z$;
\item
the commutator of $L$ acts trivially on $Z$; 
furthermore, $Z$ is an affine $T$-variety  isomorphic to $\mathbb R^r$ acted upon  by linearly independent characters of $T$.
\end{enumerate}
\end{theorem}
\qed

Let $\gamma_1,\ldots,\gamma_r$ be the characters of $T$ mentioned in Theorem \ref{thm:ACF}.
These weights are usually called the \emph{spherical roots of $\mathbf{X}$}.
The $T$-action on $Z\simeq \R^r$ is then given explicitly by 
\bq
\label{eq:3}
t \cdot z=(\gamma_1(t)z_1, \dots ,\gamma_r(t)z_r) \qquad \text{for all} \quad z=(z_1, \dots ,z_r) \in Z \quad \text{and} \quad t \in {T}.
\eq
Note that $P^u$ acts on $P^u \cdot Z$  by multiplication from the left on  $P^u$.
Further,  since $L$ normalizes $P^u$, $L$ acts on $P^u\cdot Z$ by setting
\bq
\label{eq:actions}
l \cdot (p_u,z)=(lp_ul^{-1}, lz) \in P^u\times Z\simeq P^u\cdot  Z, \qquad (p_u,z) \in P^u \times Z, \, l \in L,
\eq
 while $(L,L)$ acts trivially on $Z$.  Consequently, $P$ acts on $P^u \cdot Z$. By Theorem \ref{thm:ACF}, $P^u\cdot Z$ is an open subset of $X$ isomorphic to $ \R^s \times \R^r\simeq \R^{s+r}$, 
$P^u$ being diffeomorphic to $\R^s$ for some $s$, and  $G_0 \,  P^u\cdot Z=X$. We can therefore cover $X$ by the $G_0$-translates 
\bqn 
U_g:=g\cdot  U_e, \qquad U_e:= P^u\cdot Z, \quad g \in G_0.
\eqn
Consequently,  there exist a real-analytic diffeomorphism 
\bqn
\label{eq:16.03.2015}
\varphi:\quad \R^{s+r} \quad \longrightarrow \quad P^u \times Z \simeq P^u \cdot Z 
\eqn
and  real-analytic diffeomorphisms $\phi_g$
\bqn 
\phi_g: \quad \R^{s+r} \quad  \stackrel{\phi}{\longrightarrow} \quad   P^u \cdot Z \quad \stackrel{g}{\longrightarrow} \quad g \, P^u \cdot Z, \qquad g \in G_0,
\eqn
such that   $\mklm{(U_g,\varphi_g^{-1})}_{g\in G_0}$ constitutes an atlas of $X$.
More explicitly, if $z_j$ denotes the $j$-th coordinate function on $Z\simeq \R^r$ and $p_j$ the $j$-th  coordinate function on $P^u\simeq \R^s$, we  write
 \bq
 \label{eq:coord}
\phi_g^{-1}:  U_g \ni x \, \longmapsto \, (p_1,\dots ,p_s,z_1,\dots ,z_r)=y\in \R^{s+r}.
\eq
Note that $U_g$ is invariant under the subgroups $gTg^{-1}$ and $gP^ug^{-1}$.  In the following, $U_e$ will be called the \emph{canonical chart}.  

Next, let $g \in G_0$, $x\in U_g$, and  $h\in G_0$ be such that $h \cdot x \in U_g$. 
From the orbit structure and the analyticity of $X$  one immediately deduces 
\bq
\label{eq:chi}
z_j(h\cdot x)= \chi_j(h,x)z_j(x) , 
\eq
 where $\chi_j(h,x)$ is a function that is real-analytic in $h$ and  $x$ that  does not vanish.
We are interested in a more explicit description of the functions $\chi_j(h,x)$.

\begin{corollary} 
\label{lemma:char}
 For  any  $t\in T$, $u \in P^u$,  $x \in  U_g$, and  $j=1,\dots, r$  we have
\begin{enumerate}
\item[(a)] $z_j(gtg^{-1} \cdot x)= \chi_j(gtg^{-1},x) z_j(x) =\gamma_j(t)z_j(x)$,
\item[(b)] $z_j(gu g^{-1} \cdot x)=z_j(x)$.
 \end{enumerate}
\end{corollary}

\begin{proof}
Assertion (a) follows readily from (1) and the definition of the open sets $U_g$. Indeed, let $x=g p\cdot z\in U_g$ and $t\in T$. Then 
\bqn 
\phi_{g}^{-1}(gtg^{-1} \cdot x)=\phi^{-1}(t p\cdot z)=\phi^{-1}(t p t^{-1}, t \cdot z),
\eqn
so that the $z_j$-coordinate of $gtg^{-1} \cdot x$ reads $\gamma_j(t) z_j(x)$.  Assertion (b) is a direct  consequence of Theorem \ref{thm:ACF}.
\end{proof}

From  the classification results  of~\cite{BCF} (see precisely the list of Section 5 therein) and \cite{ressayre} (Theorem A), one immediately infers  that every wonderful $\mathbf{G}$-variety $\mathbf X$ 
whose $\mathbf T$-fixed-points are located on its closed $\mathbf G$-orbit is strict.
Even if we shall not restrict our attention to these varieties later,
we would like to close this section by mentioning that for such varieties one can construct a more refined atlas than the one given above.  
Indeed, denote by 
\bqn 
W:=N_G(T)\slash Z_G(T)
\eqn
the Weyl group of $G$ with respect to  $T$, and write $(U_w, \phi_{w}^{-1}):= (U_{n_w}, \phi^{-1}_{n_w})$ for any element $w\in W$,
$n_w\in N_G(T)$ being a representative of $w$. Note that by definition of the Weyl group  $U_w$ is independent of the representative $n_w$.
Since $ n_wTn_w^{-1}=T $,  $U_w$ carries   a natural $T$-action.  We then have the following

\begin{proposition}Suppose that $\mathbf X$ is a wonderful $\mathbf{G}$-variety such that its $\mathbf T$-fixed-points are located on its closed $\mathbf G$-orbit. Then 
$$
\mklm{(U_w,\phi^{-1}_w)}_{w\in W}
$$ 
constitutes  a finite atlas of $X$.
\end{proposition}

\begin{proof}
Let $B^-$ denote the Borel subgroup of $G$ such that $B\cap B^-=T$.
The variety $X$ has a unique projective $G$-orbit and, hence,   a unique  point  fixed by $B^-$ \cite{ACF12}.
This fixed point, denoted in the following by $y_0$,  lies in the closed $G$-orbit by assumption.
Next, let $\eta:s\mapsto (s^{a_1},..., s^{a_l})$, $a_i>0$,  be  a morphism from $\C^*$ to the algebraic torus  ${\bf T}\simeq (\C^*)^l$,
such that the  set of $\mathbf{T}$-fixed-points in $\bf{X}$ coincides with the set of fixed points of  $\mklm{\eta(s)}_{s\in \C^\ast}$  in $\bf{X}$.
By \cite{BB}, there is  a cell decomposition of $\mathbf X$ and, consequently, of  $X$ in terms of  the sets 
$$\{x\in X:\lim_{\R^\ast  \ni s\rightarrow 0}\eta(s)\cdot x=y\},
$$
where $y$ runs over the set of $T$-fixed-points of $X$. Furthermore, the open subset $P^u \cdot Z\subset X$ is given by the cell 
\bqn 
P^u \cdot Z=\{x\in X:  \lim_{\R^\ast\ni s\rightarrow 0}\eta(s)\cdot x=y_0\}, 
\eqn
see~\cite{ACF12} for details. By assumption, all $T$-fixed-points belong to the closed $G$-orbit of $X$.
On the other side, it is well-known that the $T$-fixed-points of a projective $G$-orbit are indexed by the Weyl group $W$.
More specifically, for each such $y$ there exists a $w\in W$ such that $y=n_w y_0$ for any representative $n_w\in N_G(T)$ of $w$.
Noticing that the aforementioned cells are just contained in the $W$-translates of $P^u \cdot Z$, one finally obtains the proposition.
\end{proof}

\begin{remark}
The atlas from the previous proposition is a generalization of that constructed by  Oshima for the compactification of a Riemannian symmetric space of non-compact type, see \cite{oshima78}.
\end{remark}

\section{Microlocal analysis of integral operators on wonderful varieties}
\label{Sec:int op}

With the notation as in  Section \ref{sec:wond.var}, let $\mathbf{G}$ be a connected reductive algebraic group over $\C$ of rank $l$ with split real form $(G,\sigma)$. 
Let $\mathbf{X}$ be a strict wonderful $\mathbf{G}$-variety of rank $r$ and $X$ the real locus of $\mathbf{X}$ with respect to the canonical real structure on it. 
Consider now a real differentiable $G$-vector bundle $E$ on $X$ of rank $d$ and
the corresponding regular representation of $G_0$ on the space of smooth sections $\Cinft(X,E)$ of $E$ given by
\bqn 
\pi(g) s(x) =g\cdot [s(g^{-1} \cdot x)], \qquad x \in X, \, g \in G_0, \, s \in \Cinft(X,E).
\eqn
 
  Let $(L,\Cinft (G_0))$ be the left-regular representation of $G_0$ and  $\theta$   a Cartan involution on  $\g$. With respect to the left-invariant Riemannian metric on $G_0$ given by  the modified Cartan-Killing form 
 \bqn
 \langle A,B \rangle_{\theta}:=-\eklm{A, \theta B}, \qquad A,B \in \g,
 \eqn
  we denote by  $d(g,h)$ the distance between two points $g,h \in G_0$, and set $|g|=d(g,e)$, where $e$ is the identity element of $G$.
A function $f$ on $G_0$ is said to be of \emph{of moderate growth}, if there exists a $\kappa>0$ such that $|f(g)| \leq C e^{\kappa|g|}$ for some constant $C>0$.
Let further $d_{G}$ be a Haar measure on $G$, and denote by $\U$ the universal envelopping algebra of the complexification of the Lie algebra $\g$ of $G$. 
We introduce now the \emph{Casselman-Wallach space}  of  rapidly decreasing functions on $G_0$ \cite{casselman89, wallach83}.

\begin{definition}
A function $f \in \Cinft(G_0)$ is called \emph{rapidly decreasing} if it satisfies the following condition:
For every $\kappa \geq 0$ and $H \in \U$ there exists a constant $C > 0$ such that 
	$$|dL(H)f(g)| \leq C e^{-\kappa |g|}.$$ 
The space of rapidly decreasing functions on $G_0$ will be denoted by $\S(G_0)$.
\end{definition}

\begin{remark}

1) Note that $f\in \S(G_0)$ implies that  for every $\kappa \geq 0$ and $H \in \U$ one has 
$$
dL(H)f \in \L^1(G_0,e^{\kappa|g|}d_{G}).
$$
 Indeed, let $c>0$ be such that $e^{-c|g|}\in \L^1(G_0,d_{G})$, and  $\kappa \geq 0$ and $X \in \U$ be given.
Then  $|e^{(\kappa+c)|g|} dL(X)f(g) | \leq C$ for all $g \in G_0$ and a suitable constant $C>0$, so that 
\bqn 
\big \|dL(X) f e^{\kappa|\cdot|}\big \|_{\L^1(G_0,d_{G})}\leq C \, \big \| e^{-c|\cdot|}\big \|_{\L^1(G_0,d_{G})}<\infty.
\eqn

2) If $f \in \S(G_0)$, $dR(X) f\in \S(G_0)$.
Furthermore, if one compares the space $\S(G_0)$ with the Fr\'echet spaces $\mathscr{S}_{a,b}(G_0)$ defined in \cite[Section 7.7.1]{wallach},
where $a$ and $b$ are smooth, positive, $K$-bi-invariant functions on $G_0$ satisfying certain properties,
 one easily sees  that  $a(g)=e^{|g|}$ and $b(g)=1$ satisfy the selfsame properties, except for the smoothness at $g=e$ and the $K$-bi-invariance of $a$.
Besides,  it should be noticed that the space $\S(G)$ is different from the Schwartz space introduced by Harish-Chandra \cite{casselman89}. 
In our context, the introduction of the space $\S(G_0)$ was motivated by the study of strongly elliptic operators and the decay properties 
of the semigroups generated by them   \cite{ramacher06}. 
\end{remark}

Consider next  for each $f\in \S(G_0)$  the  linear operator   
\begin{equation}
\label{eq:1703}
\pi(f):\Cinft(X,E) \longrightarrow \Cinft(X,E),\quad  (\pi(f)s)(x)=\int_{G_0} f(g)\, g\cdot [s(g^{-1}\cdot x)] \d_{G}(g)\in E_x.
\end{equation}
It becomes a continuous map when endowing $\Cinft(X,E)$ with the Fr\'{e}chet topology of uniform convergence, 
and its Schwartz kernel is a  distribution section  $\mathcal{K}_f \in  \D'(X \times X, E \boxtimes E')$, 
where $E'=E^\ast\times \Omega_{X}$, and $\Omega_X$ denotes the density bundle on $X$.
Observe that the restriction of $\pi(f) s$ to any of the  $G_0$-orbits depends only on the restriction of $s\in \Cinft(X,E)$ to that  orbit.
Let $X_0$ be an open orbit in $X$. The main goal of this section is to describe the microlocal structure of the operators $\pi(f)$, 
and characterize them as  totally characteristic pseudodifferential operators on the manifold with corners $\overline{X_0}$.
Recall that according to Melrose \cite{melrose} a continuous linear map
\bqn 
Q:\CT(M) \quad \longrightarrow \quad \Cinft(M)
\eqn
on a smooth manifold with corners $M$ is called a \emph{totally characteristic pseudodifferential operator or order $m\in \R$} 
if it can be written in local charts as an oscillatory integral
\begin{equation*}
Q_{loc}u(y):=\int  e^{i \eklm{y, \xi}} q(y,\xi)  \hat u(\xi)  \,\dbar \xi, \qquad u \in \CT(\R^{n,k}), 
\end{equation*} 
where $\hat u$ denotes the Fourier transform of $u$ and $\R^{n,k} =[0,\infty)^k \times \R^{n-k}$  the standard manifold with corners  
with $0 \leq k \leq n$ and coordinates $y=(y_1,\dots, y_k,y')$, while  $\dbar \xi=(2\pi)^{-n} \d \xi$.
The  amplitude $q$ is supposed to be of  the form $q(y,\xi)=\widetilde q(y,y_1\xi_1, \dots, y_k \xi_k, \xi')$, 
where $\widetilde q (y,\xi)$ is a symbol of order $m$ satisfying in addition the \emph{lacunary condition}
\bq
\label{eq:lacunarity}
\int e^{i(1-t)\xi_j} \widetilde q(y, \xi) \d \xi_j =0 \qquad \text{for } t<0 \text{ and } 1\leq j \leq k.
\eq
Conceptually, the algebra of totally characteristic pseudodifferential operators arises from the algebra of totally characteristic differential operators,
which is generated by the vector fields tangential to the boundary $\gd \R^{n,k}$
\bqn 
y_j \frac{\gd}{\gd y_j}, \quad 1 \leq j \leq k, \qquad \frac{\gd}{\gd y_i}, \quad k+1\leq i \leq n,
\eqn
compare also \cite[Section 18.3]{hoermanderIII}. Similarly,  if $E$ and $F$ are vector bundles over $M$, a continuous linear map
\bqn 
Q:\CT(M,E) \quad \longrightarrow \quad \Cinft(M,F)
\eqn
is called a \emph{totally characteristic pseudodifferential operator of order $m\in \R$}, if for every open subset $U\subset M$ and trivializations
\bqn 
\tau_E: E_{|U} \rightarrow U \times \R^{d_E}, \qquad \tau_F: F_{|U} \rightarrow U \times \R^{d_F},
\eqn
there is a $(d_F\times d_E )$-matrix of totally characteristic pseudodifferential operators $Q_{ij}$ of order $m$ such that
\begin{equation*}
  (\tau_F\circ (Qs)_{|U})_i=\sum_{j} Q_{ij} (\tau_E\circ s)_j, \qquad s \in \CT(U; E).
\end{equation*}
In this case, one says that $Q$ is of class $L^m_{b}$.
For a more detailed exposition on totally characteristic pseudodifferential operators the reader is referred to \cite{parthasarathy-ramacher14}.

\subsection{The toric case} 
\label{sec:toric} To make  the essential ideas  behind our approach as clear as possible, we shall first consider the simplest case,
namely  the  toric one, and  restrict ourselves to  the left-regular scalar representation.
Thus, let 
$\mathbf T=(\C^\ast)^r$ be an algebraic torus,  $\mathbf{Z}=\C^r$, and let $\bf T$ act effectively on $\bf Z$ through 
\bq
\label{eq:3toric}
t \cdot z=(\gamma_1(t)z_1, \dots ,\gamma_r(t)z_r) \qquad  z=(z_1, \dots ,z_r) \in \bf Z,  \quad   t \in {\bf T},
\eq
where the $\gamma_i(t)$ are linearly independent characters of $\bf T$,  and as such given in terms of monomials with  real coefficients.
The corresponding action of $T={\bf T}^\sigma$ on the real locus $Z=\R^r$ is given by \eqref{eq:3}.
Next, let $(\nu,\mathrm{C}_0(Z))$ be the continuous  left-regular representation of $T_0$ on the Banach space of continuous functions on $Z$ vanishing at infinity given by 
\bqn 
(\nu(t) \phi)(z) =\phi(t^{-1} \cdot z), \qquad \phi \in \mathrm{C}_0(Z), \quad t \in T_0.
\eqn
We would like to describe for each $f\in \S(T_0)$  the  continuous linear operator   
\begin{equation*}
\nu(f):C_0(Z) \supset \CT(Z) \longrightarrow \Cinft(Z) \subset \D'(Z), \qquad \nu(f)=\int_{T_0} f(t)\,\nu(t) \d_T (t)
\end{equation*}
as a pseudodifferential operator on $Z=\R^r$ using  Fourier analysis,  $d_{T}(t)=({t_1\dots t_r)^{-1}}{\d t}$ being Haar measure on $T$.
For this, let $v \in \CT(Z)$. Applying the inverse Fourier transform one computes
\begin{align*}
\nu(f) v(z) &=\int_{T_0} f(t) v(t^{-1} \cdot z) d_T(t) =\int_{T_0} \left [\int_{\R^r} e^{i \eklm{t^{-1} \cdot z,  \, \xi}} \hat v(\xi) \dbar \xi \right ] f(t) \d_T( t) \\
&= \int_{\R^r} e ^{i \eklm{z, \, \xi}} q_f(z,\xi)\hat v(\xi) \dbar\xi,
\end{align*}
where 
\bqn 
q_f(z,\xi):=e^{-i\eklm{z,\xi}} \int_{T_0} f(t) e^{i\eklm{t^{-1}\cdot z, \, \xi}} \d_T( t)
\eqn 
represents the symbol of $\nu(f)$, and constitutes a  polynomially growing function in $\xi$ for general $z$ due to  the non-transitivity of the $T$-action on $Z$.
Now,   observe that the fundamental vector fields of the $T$-action on $Z$ are given by linear combinations of  the differential operators
 \bqn 
 z_j \frac{\gd }{\gd z_j}, \qquad j=1,\dots, r,
 \eqn
 which correspond to vector fields tangential to the divisor $\mklm{z \in Z: z_1 \cdots z_r=0}$. 
Therefore, it is to be expected that   $\nu(f)$  constitutes a totally characteristic pseudodifferential operator on each of the $2^r$-tants in $\R^r$, 
for which we have to verify that 
\begin{align}
\begin{split}
\label{eq:0203}
\widetilde q_f(\xi)&:=q_f(z,\xi_1/z_1,\dots, \xi_r/z_r )
=e^{-i\eklm{(1,\dots, 1), \,\xi}} \int_{T_0} f(t) e^{i\eklm{(\gamma_1(t^{-1}),\dots, \gamma_r(t^{-1})),\,  \xi}} \d_T (t)
\end{split}
\end{align}
defines a lacunary symbol of order $-\infty$. That is, we have to show that $\widetilde q_f(\xi)$ satisfies the lacunary condition \eqref{eq:lacunarity} and
that  for any $N \in \N$ and arbitrary multi-indices $\alpha$  there exists a constant $ C_{N,\alpha}$ such that
\begin{equation}
\label{eq:0303}
  |(\gd ^\alpha_\xi  \, \widetilde  q_f) (\xi) | \leq \frac 1 {(1+|\xi|^2)^N} C_{N,\alpha} \qquad \xi \in \R^r.
\end{equation}
In this way, we are led to   the following
\begin{definition}
Let 
  \bqn
\F_{\mathrm{spher}}:\S(T_0) \ni f \mapsto  \F_{\mathrm{spher}}(f)(\xi):=\int_{T_0} f(t) e^{i\eklm{(\gamma_1(t^{-1}),\dots, \gamma_r(t^{-1})),\,  \xi}} \d_T( t) \in \Cinft(\R^r).
 \eqn
  \end{definition}
\noindent
Note that $\F_{\mathrm{spher}}$ is invariant under conjugation, that is, for arbitrary $s \in T_0$ we have
\bqn 
\F_{\mathrm{spher}}(\iota_s^\ast f)(\xi) = \int_{T_0} f(t^{-1}) e^{i\eklm{(\gamma_1(s^{-1}t^{-1}s),\dots, \gamma_r(s^{-1}t^{-1}s)),\,  \xi}} \d_T (t)=\F_{\mathrm{spher}}(f)(\xi),
\eqn 
where $\iota_s: t \mapsto sts^{-1}$ denotes  conjugation in $T_0$. Now, in order to prove that the auxiliary symbol $\widetilde q_f(\xi)=e^{-i(\xi_1+\dots +\xi_r)} \F_{\mathrm{spher}}(f)(\xi)$ satisfies \eqref{eq:0303} we will actually show that 
\bq
\label{eq:1803}
 \F_{\mathrm{spher}}:\S(T_0) \longrightarrow \S(\R^r),
 \eq
 where $\S(\R^r)$ denotes the usual Schwartz space on $\R^r$.
In the same way as one verifies that the usual Fourier transform defines a mapping from  $\S(\R^r)$ into itself, 
we shall use partial integration to do so, and exploit the fact that  in the definition of $\F_{\mathrm{spher}}(f)$ only   the transitive action of $T_0$ on $Z^\ast_+:=(\R_\ast^+)^r$ is considered.
We shall only outline here the main steps, which will be carried out in detail within the more general context of Section \ref{sec:parab}.
Thus, setting $\psi_{\xi}(t):=\eklm{(\gamma_1(t),\dots, \gamma_r(t)),\, \xi}$ one computes 
\bq
\label{eq:0603}
\frac{\gd}{\gd t_j} e^{i\psi_{\xi}(t)}=ie^{i \psi_{\xi}(t)} \sum_{i=1}^r  \Gamma_{ji}  \, \xi_i, \qquad \Gamma_{ji}={\frac{\gd \gamma_i}{\gd t_j}(t)}.
\eq
Since the matrix $\Gamma=\mklm{\Gamma_{ji}}$ is non-singular for any $t\in T$ due to the fact that the characters $\gamma_i$ are linearly independent,
one can express any polynomial in $\xi$ by a linear combination of $t$-derivatives of $e^{i\psi_{\xi}(t)}$.  We recall now the following integration formula.

\begin{proposition}
\label{prop:A}
Let $G$ be a real reductive group and $G_0$ the component of the identity. Let $f_1\in\S(G_0)$, and
assume that $f_2 \in \Cinft(G_0)$, together with all its derivatives, is of moderate growth.
Further, let $\mklm{\mathcal{X}_1,\dots, \mathcal{X}_{\dim \g}}$ denote a basis of the Lie algebra of $G$.
Then, for arbitrary multiindices $\gamma$, 
\begin{equation}
\label{A}
\int_{G_0}f_1(g) [dL(\mathcal{X}^\gamma) f_2](g) d_{G_0}(g)=(-1)^{|\gamma|} \int _{G_0}
[dL(\mathcal{X}^{\tilde \gamma}) f_1](g) f_2(g) d_{G_0}(g),
\end{equation}
where we wrote   $\mathcal{X}^\gamma=\mathcal{X}^{\gamma_1}_{i_1}\dots \mathcal{X}^{\gamma_r}_{i_r}$,   
$\mathcal{X}^{\tilde \gamma}=\mathcal{X}^{\gamma_r}_{i_r}\dots \mathcal{X}^{\gamma_1}_{i_1}$.   
\end{proposition}
\begin{proof}
See \cite[Lemma 10.22]{knapp} or  \cite[Proposition 1]{ramacher06}.
\end{proof}
By taking into account that $f$ is rapidly decreasing   and integrating $\F_{\mathrm{spher}}(f)$ according to the previous proposition we obtain \eqref{eq:1803}.
Since \eqref{eq:lacunarity} is a direct consequence\footnote{See end of Section \ref{sec:parab} for details.}  of the orbit structure of the $T$-action on $Z$, 
it follows that $\nu(f)$ is a totally characteristic pseudodifferential operator of order $-\infty$ on each of the $2^r$-tants of $Z=\R^r$.

\subsection{The parabolic case} 
\label{sec:parab}
Next, let  $P$ be the standard parabolic  subgroup of $ G$ such that $G_0/P_0$ is isomorphic to the unique closed $G_0$-orbit, and let   
${ P} ={ P}^u { L}$ be  its  Levi decomposition with $ T\subset L$. Writing $L= S \cdot (L,L)$, the Langlands decomposition of $P$ reads
\bqn 
P=MAN, \qquad  \, M:=(L,L),\quad A:=S, \quad N:=P^u\simeq \R^s
\eqn
in  standard terminology. Note that $T=A T'$, where $T'$ is the maximal torus of $M$ contained in $T$, 
and $\dim A=r$.\footnote{Notice that in Section \ref{sec:toric} we assumed  ${\bf T}$ to act effectively on $\bf Z$, so that we have  $T=A$ there.} Recall that  $P$ acts on $P^u\cdot Z$ as in \eqref{eq:actions};
in particular,
$Z$ is acted upon trivially by the commutator of $L$, while $A$ acts as in the toric case.

We shall again consider the scalar-valued case, and restrict ourselves to the description of the continuous linear operators 
\begin{equation*}
\nu(f): \CT(P^u \cdot Z) \longrightarrow \Cinft(P^u \cdot Z),  \quad \nu(f)=\int_{P_0} f(p)\,\nu(p)d_{P}(p), \quad f \in \S(P_0),
\end{equation*}
in the canonical chart $P^u \cdot Z\subset X$, 
where $ (\nu,C_0(P^u \cdot Z))$ denotes the left-regular representation of $P$ on the Banach space $C_0(P^u \cdot Z)$ of continuous functions on $P^u \cdot Z$ vanishing at infinity. In view of the local structure theorem we  identify $P^u\cdot Z$ with $ P^u \times Z \simeq \R^{s+r}$ in this subsection.
With respect to these isomorphisms  the action of $P=MAN$ on $P^u\cdot  Z$ is given by 
\begin{gather}
\begin{split}
\label{eq:paract} 
\phi^{-1}:(man)^{-1} p_u\cdot z=n^{-1} (ma)^{-1} p_u ma \,   a^{-1}  \cdot z \longmapsto (n^{-1} (ma)^{-1} p_u ma, a^{-1} \cdot z) \\ 
\longmapsto (p_1(n^{-1} (ma)^{-1} p_u ma),\dots , \gamma_1(a^{-1}) z_1,\dots, \gamma_r(a^{-1}) z_r ), 
\end{split}
\end{gather}
compare \eqref{eq:3}-\eqref{eq:coord}. Now, for a function $v \in \CT(P^u \cdot Z)$ one has
\begin{align*}
\nu(f) v(p_u\cdot z) &=\int_{P_0} f(p) v(p^{-1} p_u\cdot z) \d p=\int_{P_0} \left [\int_{\R^{s+r}} e^{i \eklm{\phi^{-1}(p^{-1} p_u \cdot z),   \, \xi}} \hat v(\xi) 
\, \dbar \xi \right ]    f(p) \d p,
\end{align*}
where $\hat v(\xi)$ denotes the Fourier transform of $v$ as a function on $P^u \cdot Z \simeq P^u \times Z \simeq \R^{s+r}$.
Introducing the phase function
\begin{align}
\begin{split}
\label{eq:phase}
\psi_{p_u\cdot z,\xi}(m,a,n):&=\eklm{\phi^{-1}(nam \, p_u \cdot z), \xi}\\ &=\eklm{\big (p_1(n (am) p_u (am)^{-1}), \dots, \gamma _1(a) z_1, \dots ,\gamma_r(a)z_r \big),  \, \xi}
\end{split}
\end{align}
and using the integration formulas for real reductive groups  we obtain for $P_0=M_0A_0N$
\begin{gather*}
\nu(f) v(p_u\cdot z)\\=\int_{M_0 \times A_0\times N} \left [\int_{\R^{s+r}} e^{i \psi_{p_u\cdot z,\xi}(m^{-1},a^{-1},n^{-1})} \hat v(\xi) \, 
\dbar \xi \right ]  f(man) a^{-2\rho} \d m \d a \d n\\ 
= \int_{\R^{s+r}} e ^{i \eklm{\phi^{-1}(p_u \cdot z), \, \xi}} q_f(p_u \cdot z,\xi)\hat v(\xi) \dbar\xi,
\end{gather*}
where we set
\begin{gather*} 
q_f(p_u \cdot z,\xi) :=e^{-i\eklm{\phi^{-1}(p_u\cdot z),\xi}} \int_{M_0 \times A_0 \times N} f(man) e^{i\psi_{p_u\cdot z,\xi}(m^{-1},a^{-1},n^{-1})} a^{-2\rho}  \d m \d a \d n.
\end{gather*}
Here $\rho \in \a^\ast$ is given by $\rho(\mathcal{A})=\frac 12 \tr (\ad \mathcal{A}_{|\n})$ and $a^{-2\rho}= \exp(-2 \rho(\mathcal{A}))$ 
if $a=\exp \mathcal{A}\in A_0$, $\mathcal{A} \in \a$, compare \cite[Section 2.4]{wallach}.
Next, consider the $z$-independent auxiliary symbol
\begin{align*}
\begin{split}
\label{eq:parsymb}
\widetilde q_f(p_u,\xi):&=q_f(p_u \cdot z,(\xi_1,\dots,\xi_s,\xi_{s+1}/z_1,\dots \xi_{s+r}/z_r))\\ &=e^{-i\eklm{(p_1,\dots, p_s,1,\dots,  1),\, \xi}}
 \int_{M_0 \times A_0 \times N} f(man) e^{i\psi_{p_u,\xi}(m^{-1},a^{-1},n^{-1}) } a^{-2\rho}  \d m \d a \d n, 
\end{split}
\end{align*}
where  
\bqn
\psi_{p_u,\xi}(m,a,n):=\psi_{p_u\cdot (1,\dots,1),\xi}(m,a,n).
\eqn
 We now arrive at the following

\begin{definition} 
\label{def:spherfourier}
Let $\psi_{p_u \cdot z,\xi}$ be defined by \eqref{eq:phase}. We then define the mapping\footnote{From our  point of view, it would be natural to call $\F_{\mathrm{spher}}$  the \emph{spherical Fourier transform of $f$}. But since this will probably lead to confusion with other transforms in literature that are called similarly, like the  \emph{spherical transform}, which is defined for  K-bi-invariant functions on a locally compact Lie group, or  the \emph{Helgason-Fourier transform}, which is defined for right $K$-invariant functions on a connected non-compact semisimple Lie group with finite center, $K$ being a maximal compact subgroup, we desisted from doing so.}
\begin{gather*} 
\F_{\mathrm{spher}}: \S(P_0) \longrightarrow \Cinft(P^u\cdot Z \times \R^{s+r}),  \\
f \longmapsto \F_{\mathrm{spher}}(f)(p_u\cdot z, \xi)=\int_{M_0 \times A_0 \times N} f(man) e^{i\psi_{p_u\cdot z,\xi}(m^{-1},a^{-1},n^{-1}) } a^{-2\rho}  \d m \d a \d n. 
\end{gather*} 
\end{definition}

Notice that $\F_{\mathrm{spher}}$ is given in terms of the spherical roots of $\bf X$ which, together with the standard parabolic subgroup $\mathbf{P}\subset \mathbf{G}$, 
are the combinatorial objects that characterize $\bf X$. Furthermore, 
$$Z^\ast \simeq \mklm{a\in A: \gamma_i(a)\neq 0 \quad  \forall \,1\leq i\leq r}.
$$
Of course, $\F_{\mathrm{spher}}$ can be written simply as an integral over $P_0$, but using the Langlands decomposition of $P_0$    the spherical roots in $\F_{\mathrm{spher}}$ become manifest. Next, we have the following crucial
\begin{proposition} 
\label{prop:spherfourier}
The transform $\F_{\mathrm{spher}}$ defines a linear map from the Casselman-Wallach space $\S(P_0)$ to the space of symbols $S^{-\infty}(P^u\cdot Z^\ast \times \R^{s+r})$,
\bqn
\F_{\mathrm{spher}}:\S(P_0) \longrightarrow S^{-\infty}(P^u \cdot Z^\ast \times \R^{s+r}),
\eqn
where $Z^\ast :=\mklm{z \in Z: z_1 \dots z_r\not=0}$.
\end{proposition}

As in Section \ref{sec:toric}, we would like to use the integration formula of Proposition \ref{prop:A} to prove Proposition \ref{prop:spherfourier}.
But now we have to consider also  the action of $AN$ on $P^u$  besides the action of $A$ on $Z$.  Indeed, in analogy to \eqref{eq:0603} one proves 

\begin{lemma}
\label{lem:16.03.2015}
Let  $\{\mathcal{N}_1,\dots,\mathcal{N}_s\}$ and  $\{\mathcal{A}_1,\dots,\mathcal{A}_r\}$ be bases for the Lie algebras $\n$ and $\a$ of $N$ and $A$, respectively.
Further, assume that $p_u \cdot z \in P^u \cdot Z^\ast$. Then 
\bq
\label{eq:0703}
\begin{pmatrix} dL(\mathcal{N}_1)e^{i\psi_{p_u\cdot z,\xi}}\\ \vdots\\ dL(\mathcal{A}_r)e^{i\psi_{p_u\cdot ,\xi}} 
\end{pmatrix} 
(m,a,n) =ie^{i\psi_{p_u\cdot z,\xi}(m,a,n)} \, \Gamma(p_u,z,m,a,n)\cdot \xi,
\eq
where
\bqn
\Gamma(p_u,z,m,a,n)=
\left(\begin{array}{cc}
\Gamma_1 & \Gamma_2 \\
 \Gamma_3   &\Gamma_4\\
 \end{array}\right) 
=
\left(\begin{array}{cccc}
 dL(\mathcal{N}_i)p_{j,p_u}(m,a,n) & \multicolumn{1}{c|}{}  & & 0 \\
&  \multicolumn{1}{c|}{} &  &\\
\cline{1-4} &   \multicolumn{1}{c|}{}  & &\\
 dL(\mathcal{A}_i)p_{j,p_u}(m,a,n) & \multicolumn{1}{c|}{} &  &dL(\mathcal{A}_i)\gamma_j(a)z_j\\
 \end{array}\right)
\eqn
 belongs to $\GL(s+r,\R)$, and we wrote $p_{j,p_u}(m,a,n)=p_j(n (am) p_u (am)^{-1})$.
\end{lemma}
\begin{proof} 
For $\A\in\a$ one computes 
\begin{align*}
dL(\A){\psi_{p_u\cdot z,\xi}}(m,a,n)&=\frac{d}{d\epsilon}\psi_{p_u \cdot z,\xi}(m,  \e{-\epsilon\A} a, n)_{|\epsilon=0}\\ 
&=\sum_{{j}=1}^{s}\xi_{j}dL(\A)p_{j,x}(m,a,n)+\sum_{{j}=s+1}^{s+r}\xi_{j}dL(\A)\gamma_j(a)z_j,
\end{align*}
and similarly for $\Ncal \in \n$, showing \eqref{eq:0703}. In particular, $\Gamma_2$ is identically zero.
To see the invertibility of the matrix $\Gamma=\Gamma(p_u,z,m,a,n)$ note that, 
like the matrix  $\mklm{\Gamma_{ij}}$ in \eqref{eq:0603}, the matrix $\Gamma_4$ is non-singular because of the linear independence of the spherical roots $\gamma_j$,
provided that $z \in Z^\ast$. Further, due  to the transitivity of the $N$-action on $P^u$, the matrix  $\Gamma_1$ is non-singular, too.  
Thus, we conclude that $\Gamma$ is non-singular.
\end{proof}

\begin{proof}[Proof of Proposition \ref{prop:spherfourier}]
Let $p_u \cdot z \in P^u \cdot Z^\ast$. 
As a  consequence of the previous lemma one can express any polynomial in $\xi$ as a linear combination of derivatives in $\n$ and $\a$ of $e^{i \psi_{p_u\cdot z,\xi}}$.
More precisely, consider the extension of $\Gamma=\Gamma(p_u,z,m,a,n)$, regarded as an endomorphism in $\C^1[\R^{s+r}_\xi]$, 
to the symmetric algebra ${\rm{S}}(\C^1[\R^{s+r}_\xi])\simeq \C[\R^{s+r}_\xi]$.  
By Lemma \ref{lem:16.03.2015} the matrix $\Gamma$ is invertible,  so its extension to $ {\rm{S}}^N(\C^1[\R^{s+r}_\xi])$ is an automorphism, too.
We regard  the polynomials $\xi_1,\dots,\xi_{s+r}$ as a basis in $\C^1[\R^{s+r}_\xi]$,  and
 denote the image of the basis vector $\xi_i$ under the endomorphism $\Gamma$ by $\Gamma \xi_i$. 
Every monomial $\xi_{i_1} \otimes \dots \otimes \xi_{i_N}\equiv \xi_{i_1} \dots \xi_{i_N}$ can then be written as a linear combination
 \begin{equation*}
   \xi^\alpha =\sum _\beta \Lambda^\alpha_\beta (p_u,z,m,a,n) \Gamma \xi_{\beta_1} \cdots \Gamma \xi_{\beta_{|\alpha|}},
 \end{equation*}
where the $\Lambda^\alpha_\beta(p_u,z,m,a,n)$ are $\Cinft$ functions on $P^u\times Z\times M\times A\times N$ which are of moderate growth in $m,a,n$. 
Taking \eqref{eq:0703} into account, a simple computation yields for arbitrary indices $\beta_1,\dots, \beta_j$ and elements $\mathcal{X}_i\in \a \oplus \n$
  \begin{align*}
\begin{split}
    i^je^{i\psi_{p_u\cdot z,\xi}(m,a,n)} & \Gamma\xi_{\beta_1} \cdots \Gamma\xi_{\beta_j}= 
dL(\mathcal{X}_{\beta_1} \cdots \mathcal{X}_{\beta_j})e^{i\psi_{p_u\cdot z,\xi}(m,a,n)}\\&+ 
\sum_{j'=1}^{j-1} \sum _{\alpha_1,\dots, \alpha_{j'}} d ^{\beta_1,\dots, \beta_j}_{\alpha_1,\dots, \alpha_{j'}} (p_u,z,m,a,n) d L(\mathcal{X}_{\alpha_1} \cdots 
\mathcal{X}_{\alpha_{j'}}) e^{i\psi_{p_u\cdot z,\xi}(m,a,n)},
\end{split}  
\end{align*}
where the coefficients $ d ^{\beta_1,\dots, \beta_j}_{\alpha_1,\dots, \alpha_{j'}}$ are smooth and of moderate growth in $m$, $a$, $n$, 
as well as independent of $\xi$.
Thus, for arbitrary $\widetilde N \in \N$ one obtains
\begin{equation}
  \label{eq:26}
e^{i\psi_{p_u\cdot z ,\xi}(m,a,n)}=  (1+|\xi|^2)^{-\widetilde N} \sum_{j=0}^{2\widetilde N} \sum_{|\alpha | =j} 
b^N_\alpha(p_u,z,m,a,n) d L(\mathcal{X}^\alpha) e^{i\psi_{p_u\cdot z,\xi}(m,a,n)},
\end{equation}
with  $\mathcal{X}^\alpha\in \mathfrak{U}(\a\oplus \n)$ and coefficients $b_\alpha^N(p_u,z, m,a,n) $  that are of moderate growth in $m, a, n$.
Similarly,  $a^{-2\rho}$ is of moderate growth.
Since $f$ is  rapidly decreasing, integrating $\F_{\mathrm{spher}}(f)(p_u\cdot z, \xi)$ by parts according to Proposition \ref{prop:A} with respect to $N\times A$ yields 
for any $\widetilde N \in \N$, any compact subset  $\mathcal{K}\subset P^u\cdot Z^\ast$, and arbitrary multi-indices $\alpha$ and
 $\beta$ the existence of  a constant $ C_{\alpha,\beta,\mathcal{K}}$ such that
\begin{equation*}
  |(\gd ^\alpha_\xi \gd ^\beta _{p_u, z} \, \F_{\mathrm{spher}}(f)) (p_u\cdot z ,\xi) | \leq \frac 1 {(1+|\xi|^2)^{\widetilde N}} C_{\alpha,\beta,\mathcal{K}} \qquad p_u\cdot z \in \mathcal{K}, \, \xi \in \R^{s+r},
\end{equation*}
thus proving Proposition  \ref{prop:spherfourier}. 
\end{proof}

\begin{remark}
The proof of Proposition \ref{prop:spherfourier} is modelled on the proof of  \cite[Theorem 4]{ramacher06} and \cite[Theorem 2]{parthasarathy-ramacher14}, where the integral transform  $\F_{\mathrm{spher}}$ was not explicitly introduced yet, but is tacitly present.
\end{remark}

From Proposition \ref{prop:spherfourier} we now infer that  
$$\widetilde q_f(p_u,\xi)=e^{-i\eklm{(p_1,\dots, p_s,1,\dots,  1),\, \xi}} \F_{\mathrm{spher}}(f)(p_u \cdot (1,\dots,1), \xi) 
$$
 defines a  $z$-independent symbol in $S^{-\infty}(P_u\cdot Z, \R^{s+r})$.
It remains to verify that it satisfies the lacunary condition \eqref{eq:lacunarity} which, as we will see, 
is a direct consequence of the orbit structure of the $P$-action on $P^u\cdot Z$.
Indeed, by the previous proposition one clearly has  
\bqn
 q_f(p_u\cdot z,\xi)=e^{-i\eklm{\phi^{-1}(p_u \cdot z),\, \xi}} \F_{\mathrm{spher}}(f)(p_u \cdot z, \xi)  \in \Syms(P^u \cdot Z^\ast \times \R^{s+r}_\xi),
 \eqn 
so that $\nu(f)_{|P^u \cdot Z^\ast}$ represents a pseudodifferential operator of class $ \L^{-\infty}$ on $P^u \cdot Z^\ast$.
Furthermore, the Schwartz kernel of $\nu(f)_{|P^u \cdot Z^\ast}$ is given by the oscillatory integral
\begin{equation*}
  \int e^{i\eklm{\phi^{-1}(p_u\cdot z)-\phi^{-1}(p_u'\cdot z'),\cdot \xi}} q_f(p_u\cdot z,\xi) \dbar \xi, \qquad z \in Z^\ast.
\end{equation*}
Now, due to the nature of the $P$-action on $P^u\cdot Z$, the restriction of $\nu(f) v$  to one of the $2^r$-tants $P^u\cdot { Z_{\pm \dots \pm}}$, 
where $Z_{\pm\dots \pm}:=\mklm{z \in Z: z_1\gtreqless 0, \dots, z_r\gtreqless 0}$, only depends on  the restriction of $v$  to  the selfsame $2^r$-tant, 
so that necessarily
\begin{equation*}
  \supp K_{\nu(f)} \subset (P^u\cdot {Z_{+ \dots +}}\times P^u\cdot { Z_{+\dots +}}) \cup \dots \cup (P^u\cdot {Z_{- \dots -}}\times P^u\cdot {Z_{-\dots -}}),
\end{equation*}
where $K_{\nu(f)}\in \D'(P^u \cdot Z \times P^u \cdot Z)$ denotes the Schwartz kernel of $\nu(f)$ 
as a continuous linear operator from $\CT(P^u \cdot Z)$  to $\Cinft(P^u \cdot Z)$.
Consequently, the integrals
\begin{equation*}
  \int e ^{i(z_j-z_j') \xi_{s+j}}  \widetilde q_f(p_u,(\xi_1,\dots,\xi_s, z_1 \xi_{s+1}, \dots ,z_r \xi_{s+r})) \, \dbar \xi_{s+j}, \qquad 1\leq j \leq r,
\end{equation*}
which are $\Cinft$-functions on $P^u \cdot Z^\ast\times P^u \cdot Z^\ast$, must vanish if $z_j$ and $z_j'$ do have different signs.
For $z,z' \in Z^\ast$ we can perform the substitutions $z_j\xi_{s+j} \mapsto \xi_j$ and write $t=z_j'/z_j-1$, thus arriving at the conditions
\begin{equation*}
\label{27a}
\int e^{-it\xi_j} \widetilde q_f(p_u,\xi) \, \dbar \xi_j =0 \qquad \text{ for } 1\leq j \leq r, \quad  t< -1, \quad p_u  \in P^u.
\end{equation*}
But these conditions no longer depend on $z$, meaning that  $\widetilde q_f(p_u,\xi)$ satisfies the lacunary condition \eqref{eq:lacunarity} on the whole chart $P^u \cdot Z$. Actually, this condition precisely encodes its orbit structure.
Thus, we have shown that $\nu(f)$ is a totally characteristic pseudodifferential operator of order $-\infty$ on the canonical chart $P^u \cdot Z$.

\subsection{The general case} After these considerations, we are ready to deal with the general case.
Thus, let  ${\bf G}$ be a connected reductive complex algebraic group and $\bf X$  a strict wonderful $\bf G$-variety of rank $r$, and consider the operators \eqref{eq:1703}.
Choose for each $x \in X$  open neighbourhoods $\Ucal_x \subset \Ucal '_x$ of $x$  contained in $U_{g}$ for some $g \in G_0$ depending on $x$.
Since  $X$ is compact, we can take a finite sub-covering of the open covering $\{\Ucal_x\}_{x \in X}$ to obtain a  finite atlas  
$\mklm{( \Ucal_{\rho}, \phi_{\rho}^{-1})}_{\rho \in R}$  on $X$, where $\phi_\rho=\phi_{g_\rho }$ for a suitable $g_\rho  \in G_0$.
In addition, assume that the subsets $\Ucal_\rho'$ have been chosen such that one has the trivializations
\bqn 
\tau_E^\rho: E_{|\Ucal_\rho'} \longrightarrow \Ucal_\rho' \times \R^d.
\eqn
One then computes  for $s \in \CT(\Ucal_\rho, E)$, $(\tau_E^\rho \circ s)(x)=(x,e_\rho(x))$,
\begin{align*}
(\tau_E^\rho \circ (\pi(f)s)_{|\Ucal_\rho})(x)&= \int_{G_0} f(h) \tau_E^\rho( h \cdot s(h^{-1} \cdot x)) \d_G(h)\\&=\Big (x, \int_{G_0} f(h) \, 
\mathcal{M}_\rho(h,h^{-1}\cdot x) [ e_\rho(h^{-1} \cdot x)] \d_G(h)\Big ),
\end{align*}
where $\mathcal{M}_\rho(h,  x)$ denotes the linear map on $\R^d$ induced by
\bqn
\R^d \simeq \{h^{-1}\cdot x\}\times \R^d \quad \overset{(\tau_E^\rho)^{-1}}{\longrightarrow}\quad E_{h^{-1}\cdot x} \quad  
\stackrel{h}{\longrightarrow} \quad E_x\quad \stackrel{\tau_E^\rho}{\longrightarrow} \quad \{x\} \times \R^d \simeq \R^d.
\eqn
We are therefore left with the task of examining the $(d\times d)$-matrix of scalar-valued integrals
\bq
\label{eq:3007}
 \int_{G_0} f(h) \, \mathcal{M}_{\rho}(h, x)_{ij} [ e_\rho(h^{-1} \cdot x)]_i \d_G(h),
\eq
where the components of $\mathcal{M}_\rho(h,x)$ and of  $e_\rho(h^{-1}\cdot x)$ are given in terms of some fixed basis of $\R^d$.
In particular, it is sufficient to consider  the scalar case, so  that we are left with the description of the  convolution operators 
\begin{equation*}
\nu(f):C(X) \supset \CT(X) \longrightarrow \Cinft(X) \subset \D'(X), \quad  \nu(f)=\int_{G_0} f(h)\,\nu(h)d_{G}(h), \quad f \in \S(G_0),
\end{equation*}
$(\nu,C(X))$ being the left-regular representation of $G$ on the Banach space $C(X)$ of continuous functions on $X$. For this, let $v\in \CT(g \, P^u \cdot Z)$ be given by $v=u \circ \phi_g^{-1}$, 
where $u \in \CT(\R^{s+r})$ and $g \in G_0$. By the unimodularity of $G_0$ one computes for $f \in \S(G_0)$
\begin{align*}
(\nu(f) v) ( g \, p_u \cdot z) &= \int_{G_0} f(h) \, v(h^{-1} g \, p_u \cdot z) \d h= \int_{G_0} f(ghg^{-1}) v( (ghg^{-1})^{-1} g\, p_u \cdot z) \d h\\ 
&=\int_{G_0} (L_{g^{-1}}R_{g^{-1}}f)(h) (u\circ \phi^{-1})(  h^{-1} p_u \cdot z) \d h.
\end{align*}
Thus,  the description of $\nu(f)$ in the chart $U_g = g \, P^u \cdot Z$  is reduced  to its study in the canonical chart 
$U_e=P^u \cdot Z \simeq P^u \times Z \simeq \R^{s+r}$.
In the analysis of the  integrals \eqref{eq:3007} we can therefore assume that $\Ucal _\rho$ is contained in the chart $P^u \cdot Z$. Let 
\bqn
G_0=KP_0=KM_0A_0N
\eqn 
 be the Cartan decomposition of $G_0$. Besides the action of $P$, which leaves  $P^u \cdot Z$ invariant, 
we have to consider now also the action of $K$, which does not  leave  $P^u \cdot Z$ invariant.
In view of \eqref{eq:paract}, for $k\in K$ close to the identity we have
 \begin{gather}
 \label{eq:1710} 
\phi^{-1}:(kman)^{-1} p_u\cdot z \stackrel{\simeq}\longmapsto (p_1(n^{-1} (ma)^{-1} p^k_u ma),\dots , \gamma_1(a^{-1}) z^k_1,\dots ), 
\end{gather}
where $k^{-1} p_u \cdot z= p_u^k \cdot z^k$ for some $(p_u^k,z^k) \in P^u \times Z$.
Let  therefore  $\alpha_\rho\in \CT(\mathcal \Ucal_\rho)$  and $\bar \alpha_\rho$ be another function satisfying $ \bar \alpha_\rho\in \CT(\mathcal{U}_\rho')$,  
$\bar \alpha_{\rho\, |\mathcal{U}_\rho}\equiv 1$. Assume that  $v=u \circ\phi^{-1} \in \CT(\mathcal{U}_\rho)$.
With the integration formulas for real reductive groups  we obtain for $p_u \cdot z \in \Ucal_\rho$
\begin{gather*}
\nu(f) v(p_u\cdot z)=\int_{G_0} f(h) (\bar\alpha_\rho v)(h^{-1} p_u\cdot z) \d_G(h) \\=\int_{K\times M_0 \times A_0\times N} 
\left [\int_{\R^{s+r}} e^{i \psi_{p^k_u\cdot z^k,\xi}(m^{-1},a^{-1},n^{-1})} \hat u(\xi) \, \dbar \xi \right ] \\ \cdot  
\bar\alpha_\rho((kman)^{-1} p_u\cdot z) f(kman) a^{-2\rho}\d k  \d m \d a \d n \\ 
= \int_{\R^{s+r}} e ^{i \eklm{\phi^{-1}(p_u \cdot z), \, \xi}} q^\rho_f( p_u \cdot z,\xi)\hat u(\xi) \dbar\xi,
\end{gather*}
where $\psi_{p^k_u\cdot z^k,\xi}$ was defined in \eqref{eq:phase} and we set
\begin{align*}
q^\rho_f(p_u \cdot z,\xi):=&e^{-i\eklm{\phi^{-1}(p_u\cdot z),\xi}} \int_{K\times M_0 \times A_0 \times N}   \bar\alpha_\rho((kman)^{-1} p_u\cdot z)f(kman) \\
 & \cdot e^{i\psi_{p^k_u\cdot z^k,\xi}(m^{-1},a^{-1},n^{-1})} a^{-2\rho}  \d k \d m \d a \d n  .
\end{align*}
To  characterize $\nu(f)$ as a totally characteristic pseudodifferential operator on $\mathcal{U}_\rho$ we consider  the auxiliary symbol
\bqn 
\widetilde q^\rho_f(p_u \cdot z,\xi):=q^\rho_f(p_u \cdot z,(\xi_1,\dots,\xi_s,\xi_{s+1}/z_1,\dots \xi_{s+r}/z_r))
\eqn
and note that in terms of the integral  transform  introduced in Section \ref{sec:parab}  the symbol  $\widetilde q^\rho_f(p_u \cdot z,\xi)$ equals
\bqn
e^{-i\eklm{(p_1,\dots,p_s,1,\dots,1),\xi}} \int_K \F_{\mathrm{spher}}\big ( L_{k^{-1}} (f \overline A_{\rho,p_u \cdot z})\big )\big (p_u \cdot (\chi_1(k,p_u \cdot z), \dots,
 \chi_r(k,p_u \cdot z)),\xi\big) \, dk,
\eqn
where we wrote $\overline A_{\rho,p_u\cdot z}(h):=\overline \alpha_\rho(h^{-1} p_u\cdot z)$, and took into account 
\eqref{eq:chi}, by which $z_j^k=\chi_j(k,p_u\cdot z) z_j$.  If we now apply Proposition \ref{prop:spherfourier}, 
we see that 
\bqn
\F_{\mathrm{spher}}( L_{k^{-1}} (f \overline A_{\rho,p_u \cdot z}))(p_u \cdot (\chi_1(k,p_u \cdot z), \dots, \chi_r(k,p_u \cdot z)),\xi) \, \text{is rapidly decaying in $\xi$,}
\eqn 
 since $(\chi_1(k,p_u \cdot z), \dots, \chi_r(k,p_u \cdot z))\in Z^\ast$.
 Integrating over $K$ then yields the desired statement $\widetilde q^\rho_f(p_u \cdot z,\xi) \in S^{-\infty}(P_u \cdot Z, \R^{s+r})$, 
everything being absolutely convergent. 
Finally, the argument at the end of Section \ref{sec:parab} that showed  that $\widetilde q_f(p_u,\xi)$ satisfies the lacunarity condition 
\eqref{eq:lacunarity} also 
proves that $\widetilde q^\rho_f(p_u \cdot z,\xi)$ is lacunary. Thus, we have shown  the main result of this section. 

\begin{theorem}
\label{thm:2503}
Let $\mathbf{G}$ be a connected reductive algebraic group over $\C$  and $(G,\sigma)$ a split real form of $\mathbf{G}$.
Let $X$ be the real locus of a strict wonderful ${\bf G}$-variety $\textbf X$,  $E$ a smooth real $G$-vector bundle over $X$, 
and $(\pi,\Cinft(X,E))$ the regular representation of $G_0$.  Let $X_0$ be an open $G_0$-orbit in $X$ and $f \in \S(G_0)$.  
Then the  continuous linear operator
\begin{equation*}
\pi(f)_{|\overline{X_0}}:\CT(\overline{X_0},E) \longrightarrow \Cinft(\overline{X_0},E),
\end{equation*}
 is a  totally characteristic pseudodifferential operator of class $\L^{-\infty}_b$ on the manifold with corners $\overline{X_0}$. 
\end{theorem}
\qed

\begin{remark}
Note that  if in the previous theorem $X_0$ is a Riemannian symmetric space, then  its closure $\overline{X_0}$ in $X$ is the maximal Satake compactification of $X_0$,
see \cite[Remark II.14.10]{borel-ji}. 
\end{remark}

For later computations, we will require explicit descriptions of the kernel of $\pi(f)$ in  the different charts of $X$.
Thus,  let $\mklm{\alpha_\rho}_{\rho\in R}$ be a partition of unity subordinate to the  atlas  $\mklm{( \Ucal_{\rho}, \phi_{\rho}^{-1})}_{\rho \in R}$ and
let $\mklm{\bar \alpha_\rho}_{\rho\in R}$ be another set of functions satisfying $ \bar \alpha_\rho\in \CT(\Ucal_\rho')$,  $\bar \alpha_{\rho|\Ucal_\rho}\equiv 1$.
Fix a chart $\Ucal_\rho  \subset g \, P^u \cdot Z$ with $g \in G_0$, and  let $v\in \CT(\Ucal_\rho)$  be given by $v=u \circ \phi_g^{-1}$,  
where $u \in \CT(\R^{s+r})$. We now consider the localization of the integrals \eqref{eq:3007} 
 \begin{align*}
(\,^{ij}Q_{f}^{\rho} u)(y)&:= \int_{G_0} f(h) \, \mathcal{M}_{\rho}(h,gp_u \cdot z)_{ij} (\bar \alpha_\rho v) (h^{-1} gp_u \cdot z) \d_G(h)\\
&=\int_{G_0} f(ghg^{-1})  c^{ij}_{\rho}(g,h, p_u \cdot z) (u\circ \phi^{-1})(  h^{-1} p_u \cdot z) \d_G(h),
 \end{align*}
where we wrote $y=(p,z)=\phi^{-1}_g(g p_u \cdot z)=\phi^{-1}( p_u \cdot z)$ and  put   $c^{ij}_{\rho}(g,h, p_u \cdot z):=\mathcal{M}_{\rho}(ghg^{-1}, g p_u \cdot z)_{ij} \bar 
\alpha_\rho  (g h^{-1} p_u \cdot z)$. 
If we now define
\begin{gather*}
^{ij} q^\rho_f(y,\xi):=e^{-i\eklm{y,\xi}} \int_{K\times M_0 \times A_0 \times N} c^{ij}_{\rho}(g,kman, p_u \cdot z)   f(gkmang^{-1}) \\  
\cdot e^{i \psi_{p^k_u\cdot z^k,\xi}(m^{-1},a^{-1},n^{-1}) } a^{-2\rho} \d k \d m \d a \d n
\end{gather*}
we obtain
\begin{equation}
\label{eq:15.3.2015}
(\,^{ij}Q_{f}^{\rho}) u(y)= \int_{\R^{s+r}} e ^{i \eklm{y,\xi}} \,\, ^{ij} q^\rho_f(y,\xi) \, \hat u(\xi) \dbar\xi.
\end{equation}
By our previous considerations in this subsection,  
\begin{align}
\begin{split}
\label{eq:auxsym}
^{ij}\widetilde q^\rho_f(y,\xi):=&\,^{ij} q^\rho _f(y,(\xi_1,\dots,\xi_s,\xi_{s+1}/z_1,\dots, \xi_{s+r}/z_r))\\ =&e^{-i\eklm{(y_1,\dots,y_s, 1, \dots, 1),\, \xi}} 
\int_{K\times M_0 \times A_0 \times N} c^{ij}_{\rho}(g,kman, p_u \cdot z)   f(gkmang^{-1}) \\  &\cdot e^{i \eklm{(p_1(n^{-1} (ma)^{-1} p^k_u (ma)), \dots, 
\gamma _1(a^{-1}) \chi_1(k,p_u\cdot  z) , \dots ),  \, \xi}} a^{-2\rho} \d k \d m \d a \d n
\end{split}
 \end{align}
is a lacunary symbol of order $-\infty$. Further, for $f \in \S(G_0)$, the restriction of $\pi(f)$ to $\Ucal_\rho$ is  given by 
the $(d\times d)$-matrix of operators 
\eqref{eq:15.3.2015}.
In particular, the kernel of $\pi(f)$ is determined by its restriction to $\{ y=(p,z) \in \R^{s+r}: z_1 \cdots z_r \not=0\}\times \{ y=(p,z) \in \R^{s+r}: z_1 \cdots z_r \not=0\}$, and given  by the matrix of oscillatory integrals
\bqn 
\left( \begin{array}{ccc}
  K_{\, ^{11}Q_f^\rho}(y,y') & \dots &   K_{\, ^{1d}Q_f^\rho}(y,y') \\
\vdots & \ddots & \vdots \\
  K_{\, ^{d1}Q_f^\rho}(y,y')  & \dots &   K_{\, ^{dd}Q_f^\rho}(y,y') 
\end{array}
\right), \qquad y_{s+1}\cdots y_{s+r} \not=0,
\eqn
where
\begin{align}
\label{eqn: explicit form}
\begin{split}
  K_{\, ^{ij}Q_f^\rho}(y,y') &=  \int_{\R^{s+r}} e ^{i\eklm{y-y',   \xi}} \, \, ^{ij}q^\rho _f (y,\xi) \dbar \xi\\ &=\frac 1 {|y_{s+1}\cdots y_{s+r}|}\int_{\R^{s+r}} e^{i\eklm{y-y',   
(\xi_1,\dots,\xi_s,\xi_{s+1}/y_{s+1},\dots, \xi_{s+r}/y_{s+r})}} \, \, ^{ij} \tilde q_f^\rho(y,\xi)  \dbar \xi\\
&=\frac 1 {|y_{s+1}\cdots y_{s+r}|} \,^{ij}\widetilde Q_f^\rho\Big (y,y_1-y'_1, \dots, 1 - \frac{y'_{s+1}}{y_{s+1}}, \dots\Big ),  
\end{split}
\end{align}
and   $^{ij}\widetilde Q_f^\rho(y,\cdot )$ denotes the inverse Fourier transform of the lacunary symbol $^{ij}\tilde q_f^\rho(y,\cdot)$. In particular, \eqref{eqn: explicit form} shows that the kernel of $\pi(f)$ is smooth outside any neighborhood of the diagonal. 
The restriction of the kernel of each of the operators $^{ij}Q_f^{\rho}$ to the diagonal is given by
$$K_{\, ^{ij}Q_f^\rho}(y,y)=\frac 1 {|y_{s+1}\cdots y_{s+r}|}
\, ^{ij} \widetilde Q_f^\rho(y,0), \qquad y_{s+1}\cdots y_{s+r}
\not=0.$$ 
These restrictions yield a family of smooth functions $^{ij}\kappa_f^\rho(x):=K_{\,^{ij}Q_f^\rho}(\phi^{-1}_{\rho}(x),\phi^{-1}_{\rho}(x))$,
which define a density $^{ij} \kappa_f$ on  the union of the open $G_0$-orbits
on $X$. Nevertheless, the functions $^{ij}\kappa_f^\rho(x)$ are not
locally integrable on all of $X$, so that we cannot define a trace of $\pi(f)$ by integrating 
\bqn 
\Tr \left( \begin{array}{ccc}
 ^{11} \kappa_f& \dots &   ^{1d} \kappa_f \\
\vdots & \ddots & \vdots \\
  ^{d1} \kappa_f & \dots &   ^{dd} \kappa_f
\end{array}
\right)
\eqn
 over the diagonal $\Delta_{X
\times X} \simeq X$.
Instead, the explicit form of the local kernels  \eqref{eqn: explicit form} suggests a natural regularization for the trace of the integral operators $\pi(f)$, 
which will be accomplished in the next section. 

\section{Regularized traces and fixed point formulae}

\subsection{Regularized traces}

Let the notation be as in the previous sections. 
In the following, we shall define a regularized trace for the convolution operators \eqref{eq:1703}, based on  the explicit description  \eqref{eqn: explicit form} 
of their kernels and  a classical result of Bernstein-Gelfand on the meromorphic continuation of complex powers. 

\begin{proposition}
\label{prop:4}
Let  $\{\alpha_{\rho}\}$ be a partition of unity subordinate to the atlas $\{( \Ucal_{\rho},
\phi^{-1}_{\rho})\}_{\rho \in R}$.
Let $f \in \S(G_0)$, $\zeta  \in \C$, and define for $\Re \zeta>0$ 
\begin{align*}
\Tr_\zeta \pi(f):&= \sum_{j=1}^d \sum _\rho \int _{\R^{s+r}}  |y_{s+1} \cdots
y_{s+r}|^\zeta 
(\alpha_\rho \circ \phi_{\rho})(y)\, ^{jj}\widetilde Q_f^\rho(y,0) dy\\
&= \eklm{|y_{s+1} \cdots y_{s+r}|^\zeta,\sum_{j=1}^d \sum_\rho
(\alpha_\rho \circ \phi_{\rho}) \, ^{jj}\widetilde
Q_f^\rho(\cdot,0)}.
\end{align*}
Then $\Tr_\zeta \pi(f) $ can be continued analytically to a 
meromorphic function in $\zeta$ with at most poles at  $-1, -3, \dots$.
Furthermore, for $\zeta \in \C-\mklm{-1,-3,\dots}$,
\begin{align*}
\Theta_\pi^\zeta:\CT(G_0) \ni f \mapsto \Tr_\zeta \pi(f) \in \C
\end{align*}
defines a distribution density on $G_0$.
\end{proposition}

\begin{proof}
The proof  is analogous to the proof of  \cite[Proposition 4]{parthasarathy-ramacher14}.
In particular, the fact that $\Tr_\zeta \pi(f)$ can be continued meromorphically  is a consequence of the analytic continuation of
$|y_{s+1} \cdots y_{s+r}|^\zeta$ as a distribution in $\R^{s+r}$ \cite{bernstein-gelfand69}.

\end{proof}
 
Consider next the Laurent expansion of $\Theta_\pi^\zeta(f)$ at $\zeta=-1$.
For this, let $u \in \CT(\R^{s+r})$ be a test function,
and consider the expansion
\bqn 
\eklm{|y_{s+1} \cdots y_{s+r}|^\zeta,u }= \sum_{j=-J}^\infty
S_j(u) (\zeta+1)^j,
\eqn
where $S_j \in \D'(\R^{s+r})$. Since $|y_{s+1} \cdots
y_{s+r}|^{\zeta+1}$ has no pole at $\zeta=-1$, we necessarily
must have
\bqn 
|y_{s+1} \cdots y_{s+r}| \cdot S_j =0 \quad \text{for }\,
j<0, \qquad |y_{s+1} \cdots y_{s+r}| \cdot S_0=1
\eqn
as distributions. In other words, $S_0 \in \D'(\R^{s+r})$
represents a distributional inverse of $|y_{s+1} \cdots y_{s+r}|$. Thus,   we  arrive at the main result of this paper.
\begin{theorem}
\label{prop:tracereg}
For $f \in \S(G_0)$,  let  the regularized trace of the operator $\pi(f)$ be defined by
\begin{align*}
\Tr_{reg}\pi(f)&:=\left\langle S_0,\sum_{j=1}^d\sum_{\rho}
(\alpha_{\rho}\circ \phi_{\rho} ) \, ^{jj} \widetilde
Q_f^\rho(\cdot,0) \right\rangle.
\end{align*}
Then $\Theta_\pi: \CT(G_0) \ni f \mapsto \Tr_{reg} \pi(f) \in \C$  constitutes  a distribution density on  $G_0$ which is  given in terms of the spherical roots of ${\bf X}$. It is called  the \emph{character of the representation $(\pi,\Cinft(X,E))$}. 
\end{theorem}
\qed

\begin{remark} \hspace{0cm} 
\begin{enumerate}
\item The coordinate invariance of the defined regularized trace $\Tr_{reg}\pi(f)$  is guaranteed by standard arguments, see \cite[Corollary 1]{atiyah70}.
\item Alternatively, a similar regularized trace can be defined using the calculus of b-pseudodifferential operators developed by Melrose.
For a detailed description, the reader is referred to  \cite[Section 6]{loya98}.
\end{enumerate}
\end{remark}

In what follows, we shall identify  distributions with distribution densities on $G$ via the Haar measure $d_G$.

 \subsection{Fixed point formulae}
\label{Sec:6}

Our next aim is to understand the distributions $\Theta^\zeta_\pi$ and $\Theta_\pi$ in terms of the $G_0$-action on $X$.
 We shall actually show that on a certain open set of transversal elements, they are represented by locally integrable functions  given in terms of fixed points.
Similar expressions where derived by Atiyah and Bott \cite{atiyah-bott68} for the global character of an induced representation of $G_0$, and 
we were inspired by these formulae.

Let the notation be as before, and consider for each element   $g\in G$ the transformation  $\Phi_g: X \rightarrow X, \, x \mapsto g^{-1} \cdot x$. Recall that $\Phi_g$ is called \emph{transversal} 
if all its fixed points are \emph{simple}, meaning that one has $\det(\1 - (d\Phi_g)_{x_0})\not=0$ at each fixed point $x_0\in X$.
Further note that the set $G_0(X):=\mklm{g \in G_0: \, \Phi_g \, \text{is transversal}}\subset G_0$ of elements  acting transversally on $X$ is  open.
With the notation as before we then have the following 

\begin{theorem}
\label{thm:FPF}  Let  $f \in \CT(G_0)$ have  support in $G_0(X)$, and $\zeta \in \C$ be such that  $\Re \zeta>-1$.
Let further $\Fix(X,h)$ denote the set of fixed points  on $X$ of $\Phi_h$,  $h\in G$.
Then $\Tr_\zeta \pi(f)$ is given by the expression
\begin{align*}
\begin{split}
\Tr_\zeta \pi(f)=&\int_{G_0(X)}  f(h) \left (\sum_{x \in \Fix(X,h)} \frac{\Tr(h:E_x \rightarrow E_x)}{ |\det (\1 -d\Phi_{h}(x))|} \right. \\
& \left. \cdot  \sum _\rho  { \alpha_\rho (x)  |y_{s+1}({g_\rho }^{-1} \cdot x) \cdots
y_{s+r}({g_\rho }^{-1} \cdot x)|^{\zeta+1} } \right ) d_G(h).
\end{split}
\end{align*}
  In particular, $\Theta_\pi^\zeta:\CT(G_0) \ni f \to \Tr_\zeta \pi(f)\in \C$ is regular on $G_0(X)$. 
\end{theorem}

\begin{proof}
The proof is similar to the proof of Theorem 7 in \cite{parthasarathy-ramacher14}. By Proposition \ref{prop:4}, 
\begin{align*}
\Tr_\zeta \pi(f)&=\sum_{j=1}^d\sum _\rho \int _{\R^{s+r}} |y_{s+1} \cdots
y_{s+r}|^{\zeta} 
(\alpha_\rho \circ \phi_{\rho})(y)   \, ^{jj} \widetilde Q_f^\rho(y,0) dy
\end{align*}
is a meromorphic function in $\zeta$ with possible poles at $-1,-3,\dots$. Assume that $\Re \zeta>-1$. Since 
$
^{ij} \widetilde Q_f^\rho(y,0)=\int \,^{ij} \tilde q_f^\rho(y,\xi) \dbar \xi,
$
where $ ^{ij}\tilde q_f^\rho(y,\xi)\in \Sym^{-\infty}_{la}(\R^{s+r} \times \R^{s+r})$ is rapidly decaying in $\xi$, 
the order of integration can be interchanged, yielding
\begin{align*}
\Tr_\zeta \pi(f)&=\sum_{j=1}^d \sum _\rho \int_{\R^{s+r}}  \int _{\R^{s+r}} |y_{s+1} \cdots
y_{s+r}|^{\zeta}
(\alpha_\rho \circ \phi_{\rho})(y)  \, ^{jj}\tilde q_f^{\rho}(y,\xi)  dy \, \dbar \xi.
\end{align*}
Let $\chi \in \CT(\R^{s+r}, \R^+)$ be equal $1$ in a neighborhood of $0$ and $\epsilon >0$. Then, by Lebesgue's theorem on bounded convergence,
\bqn 
\Tr_s \pi(f)=\lim_{\epsilon \to 0} I_\epsilon,
\eqn
where we set
\bqn
I_\epsilon:= \sum_{j=1}^d \sum _\rho \int_{\R^{s+r}}  \int _{\R^{s+r}} |y_{s+1} \cdots y_{s+r}|^{\zeta}
(\alpha_\rho \circ \phi_{\rho})(y)  \,^{jj} \tilde q_f^{\rho}(y,\xi)  \chi(\epsilon \xi) \, dy  \, \dbar \xi.
\eqn
In what follows, write  $\phi^h(y):=(\phi^{-1} \circ h^{-1} \circ \phi)(y)$, and let $T_y$ be the diagonal $(r\times r)$-matrix with entries $y_{s+1},\dots, y_{s+r}$.
Interchanging the order of integration once more  one obtains  with 
 \eqref{eq:auxsym} 
\begin{align*}
I_\epsilon=&\int_{G_0}   \sum_{j=1}^d \sum _\rho   f(g_\rho hg^{-1}_\rho) \int_{\R^{s+r}} \int _{\R^{s+r}}e^{i \eklm{\Psi(h,y),  \xi}} \\ 
& \cdot  c^{jj}_{\rho}(g_\rho ,h, \phi(y))   (\alpha_\rho \circ \phi_{\rho})(y) |y_{s+1} \cdots y_{s+r}|^{\zeta} \chi(\epsilon \xi) dy \,  \dbar \xi \,  d_G(h),
\end{align*}
where  with  $\phi(y)=p_u \cdot z$ we wrote
\begin{gather*}
\Psi(h,y):=[(\1_s\otimes T_y^{-1})(\phi^h(y)-y)]=\big (y_1(h^{-1}p_u \cdot z)-y_1,\dots,\chi_{{}_1}(h^{-1},p_u \cdot z)-1,\dots, \big ),
\end{gather*}
everything being absolutely convergent. Let us now write $I_\epsilon(h)$ for the integrand of the $G_0$-integral in $I_\epsilon$, so that 
$
I_\epsilon= \int_{G_0} I_\epsilon(h) \, d_G(h). 
$
In order  to pass to the limit under the integral, we shall show that $\lim_{\epsilon \to 0} I_\epsilon(h)$ is an integrable function on $G_0$.
Now, it is not difficult to see that, as $\epsilon \to 0$, the main contributions to $I_\epsilon(g)$ originate from the fixed points of $\Phi_h$. Since 
\bq
\label{eq:fix}
g \cdot x \in \Fix(X,h) \quad \Longleftrightarrow \quad x \in \Fix(X,g^{-1}hg),
\eq
it is sufficient to examine them in the canonical chart. To compute these contributions, note that due to the fact that all fixed points are simple,
$y \mapsto \phi^{h}(y)-y$ defines a diffeomorphism near fixed points.
Performing successively the changes of variables $y'=y-\phi ^h(y)$ and $y''=(\1_s \otimes T^{-1}_{y(\eps y')})y'$ one obtains  for $\lim_{\epsilon \to 0} I_\epsilon(h) $ the expression 
\begin{gather*}
 \sum _\rho    f(g_\rho h g_\rho^{-1})   \sum_{j=1}^d \sum_{x \in \Fix(P^u \cdot Z,h)} 
\frac{\alpha_\rho(g_\rho \cdot x)\mathcal{M}_{\rho}(g_\rho hg_\rho ^{-1}, g_\rho \cdot x)_{jj} |y_{s+1}(x) \cdots y_{s+r}( x)|^{\zeta+1} }{ |\det (\1 -d\Phi_{h}(x))|},
\end{gather*}
where we took into account that  for $x =\phi(y)=p_u \cdot z \in \Fix(P^u \cdot Z,h)$ one has
\begin{align*}
 c^{ij}_{\rho}(g_\rho,h,\phi(y))(\alpha_\rho \circ \phi_{\rho})(y)   &= \alpha_\rho(g_\rho \cdot x)\mathcal{M}_{\rho}(g_\rho hg_\rho ^{-1}, g_\rho \cdot x)_{ij}.
\end{align*}
The limit function $\lim_{\epsilon \to 0} I_\epsilon(g)$ is therefore clearly integrable on $G(X)$ for $\Re \zeta > -1$.
Passing to the limit under the integral and conjugating then yields with \eqref{eq:fix}
\begin{gather*}
\Tr_\zeta \pi(f)=\lim_{\epsilon \to 0} I_\epsilon= \lim_{\epsilon \to 0} \int_{G_0} I_\epsilon(h) \, d_G(h)\\ =\int_{G_0}  f(h) \sum_{x \in \Fix(X,h)} \sum_{j=1}^d  
\sum _\rho        \frac{ \alpha_\rho (x) \mathcal{M}_\rho(h,x)_{jj} |y_{s+1}({g_\rho^{-1} } \cdot x) \cdots
y_{s+r}({g_\rho^{-1} } \cdot x)|^{\zeta+1} }{ |\det (\1-d\Phi_{h}(x))|} \, d_G(h).
\end{gather*}
Since $\sum_{j=1}^e \mathcal{M}_{jj}(g,x)= \Tr(g:E_x \rightarrow E_x)$, the assertion of the theorem  follows.  
 \end{proof}
From the previous theorem it is now clear that if $f 
\in \CT(G_0(X))$,  $\Tr_\zeta \pi(f)$ is not singular at $\zeta=-1$. Consequently, we obtain

\begin{corollary} 
\label{cor:2}
Let $f \in \CT(G_0)$ have support in $G_0(X)$. Then 
\bqn
\Tr_{reg} \pi(f)= \Tr_{-1} \pi(f)= \int_{G_0(X)}   f(g)  \sum_{x \in \Fix(X,g)}     \frac{\Tr(g:E_x \rightarrow E_x) }{ |\det (\1-d\Phi_{g}(x))|} \, d_G(g).
\eqn
In particular, the distribution $\Theta_\pi:f \to \Tr_{reg}(f)$ is regular on $G_0(X)$. 
\end{corollary}
\begin{proof}
By Theorem \ref{thm:FPF}, $\Tr_\zeta \pi(f)$ has  no pole at $\zeta=-1$. Therefore, the Laurent expansion of $\Theta_\pi^\zeta(f)$ at $\zeta=-1$ must read
\begin{align*}
\Tr_\zeta \pi(f)&= \eklm{|y_{s+1} \cdots y_{s+r}|^\zeta,\sum_\rho \sum_{j=1}^d
(\alpha_\rho \circ \phi_{\rho}) \, ^{jj} \widetilde
Q_f^\rho(\cdot,0)}\\ & = \sum_{j=0}^\infty
S_j\Big (\sum_\rho \sum_{j=1}^d
(\alpha_\rho \circ \phi_{\rho})\, ^{jj} \widetilde
Q_f^\rho(\cdot,0)\Big ) (\zeta+1)^j,
\end{align*}
where $S_j \in \D'(\R^{s+r})$.  Thus, 
\bqn 
\Tr_{-1} \pi(f)=\eklm{S_0,\sum_\rho \sum_{j=1}^d
(\alpha_\rho \circ \phi_{\rho})\, ^{jj} \widetilde
Q_f^\rho(\cdot,0)}=\Tr_{reg} \pi(f),
\eqn
and the assertion follows with the previous theorem. 
\end{proof}

Note that from  Corollary \ref{cor:2}  it is immediate  that $\Theta_\pi$ is independent of the chosen atlas of $X$ and invariant under conjugation  as a distribution  on $G_0(X)$.
Furthermore,   a flat trace     $\Tr^\flat \pi(g)$ of  $\pi(g) $ can be defined, and as it turns out \cite{atiyah-bott67}, 
\bqn 
\Tr^\flat \pi(g)=  \sum_{x \in \Fix(X,g)}     \frac{\Tr(g:E_x \rightarrow E_x) }{ |\det (\1-d\Phi_{g}(x))|},
\eqn
so that we finally obtain 
\bqn 
\Theta_\pi(f)=\Tr_{reg} \pi(f)=\int_{G_0(X)} f(g) \Tr^\flat \pi(g) d_G(g), \qquad f \in \CT(G_0(X)).
\eqn
We would like to close by noting that on $G_0(X)$ the distribution $\Theta_\pi$ no longer explicitly  depends on the spherical roots of $\bf X$, but it of course still does on the whole group $G_0$.


\providecommand{\bysame}{\leavevmode\hbox to3em{\hrulefill}\thinspace}
\providecommand{\MR}{\relax\ifhmode\unskip\space\fi MR }
\providecommand{\MRhref}[2]{%
  \href{http://www.ams.org/mathscinet-getitem?mr=#1}{#2}
}
\providecommand{\href}[2]{#2}


\end{document}